\documentclass{amsart}

\usepackage{IEEEtrantools}
\usepackage{enumerate}
\usepackage{amssymb}

\newtheorem{theorem}{Theorem}
\newtheorem{definition}[theorem]{Definition}
\newtheorem{example}[theorem]{Example}
\newtheorem{question}[theorem]{Question}
\newtheorem{lemma}[theorem]{Lemma}

\newtheorem{proposition}[theorem]{Proposition}
\newtheorem{notation}[theorem]{Notation}
\newtheorem{remark}[theorem]{Remark}
\newtheorem{claim}[theorem]{Claim}

\newcommand{\er}{\mathbb R}
\newcommand{\en}{\mathbb N}
\newcommand{\N}{\mathbb N}

\newcommand{\zet}{\mathbb Z}

\newcommand{\diam}{\operatorname{diam}}

\def\rest{\hskip-2.5pt\restriction}

\begin{document}

\title[Haar meager sets revisited]{Haar meager sets revisited}

\author{Martin Dole\v{z}al}
\author{Martin Rmoutil}
\author{Benjamin Vejnar}
\author{V\'aclav Vlas\'ak}
\thanks{
	The first author was supported by RVO: 67985840.
	The second author has received funding from the European Research  Council under the European Union's Seventh Framework Programme (FP/2007-2013) / ERC Grant Agreement n.291497.
		The third author was supported by the grant GA\v CR 14-06989P and he is a junior researcher in the University Center for Mathematical Modeling, Applied Analysis and Computational Mathematics (Math MAC).
		The fourth author was supported by the grant GA\v CR P201/15-08218S and he is a junior researcher in the University Center for Mathematical Modeling, Applied Analysis and Computational Mathematics (Math MAC)}

\begin{abstract}
In the present article we investigate Darji's notion of Haar meager sets from several directions. We consider alternative definitions and show that some of them are equivalent to the original one, while others fail to produce interesting notions. We define Haar meager sets in nonabelian Polish groups and show that many results, including the facts that Haar meager sets are meager and form a $\sigma$-ideal, are valid in the more general setting as well. The article provides various examples distinguishing Haar meager sets from Haar null sets, including decomposition theorems for some subclasses of Polish groups. As a corollary we obtain, for example, that $\zet^\omega$, $\er^\omega$ or any Banach space can be decomposed into a Haar meager set and a Haar null set. We also establish the stability of non-Haar meagerness under Cartesian product.
\end{abstract}

\maketitle

\section{Introduction}

The notion of Haar meager sets in abelian Polish groups was defined by Darji in \cite{Darji} as a topological counterpart to Haar null sets defined by Christensen in \cite{Christensen}. Let $G$ be an abelian Polish group. A set $A\subset G$ is said to be \emph{Haar null} if there exists a Borel set $B\supset A$ and a Borel probability measure $\mu$ on $G$ such that $\mu(x+B)=0$ for each $x \in G$. By Darji's definition, $A$ is said to be \emph{Haar meager} if there exists a Borel set $B\supset A$, a compact metric space $K$ and a continuous (``witness'') function $f:K\to G$ such that $f^{-1}(x+B)$ is meager in $K$ for each $x\in G$.

Among other results, Christensen proved in his paper that Haar null sets form a $\sigma$-ideal (i.e. are closed under subsets and countable unions) and that they coincide with sets of Haar measure zero in locally compact abelian Polish groups. Similarly, Darji proved in his paper that Haar meager sets form a $\sigma$-ideal contained in the $\sigma$-ideal of meager sets and that they coincide with meager sets in locally compact abelian Polish groups. 

One might suspect that the similarity between Haar null and Haar meager sets is only formal and does not go much beyond the very basic properties. It turns out, however, that the two notions share other properties as well, which resembles the duality between category and measure. Indeed, many of the results about Haar meager sets (as well as their proofs) were inspired by analogous results about Haar null sets. In particular, in non-locally compact abelian Polish groups, it is known that compact sets are both Haar null (essentially due to Christensen, \cite[Theorem~2]{Christensen}; see also \cite{Dougherty}) and Haar meager (see \cite{Jablonska}), and, as expected, the reason is the same in both cases: If $A$ is a Borel subset of an abelian Polish group which is non-Haar null or non-Haar meager, then $A-A$ contains an open neighborhood of $0$. (To complete the argument, it suffices to observe that if $A$ is compact, so is $A-A$.)

In this paper we continue the study of Haar meager sets from several viewpoints. In Section~2 we extend the definition of Haar meager sets to all Polish groups (not necessarily abelian) and prove that in any Polish group they are meager and form a $\sigma$-ideal. In the following sections, some of our results require the commutativity of the group operation, others do not. 

In Section~3 we examine the definition of Haar meagerness itself; more precisely, we look at some other natural candidates for the definition of a topological analogue of Haar null sets. We discuss two natural ways how to modify the definition: We can alter the requirements on the `hull' $B$ or we can consider removing the witness function from the definition. An example of alternative definition is the notion of \emph{naively Haar meager sets} where we remove the hull $B$ from the definition of Haar meagerness. (This is a notion analogous to that of \emph{naively Haar null} sets introduced in \cite{ElekesVidnyanszky2015}.) We prove that in uncountable abelian Polish groups there are always naively Haar meager sets which are not Haar meager; moreover, under the Continuum Hypothesis we show that the group can be decomposed into two naively Haar meager sets.

Section~4 establishes the stability of non-Haar meagerness under Cartesian product. We also provide an example which shows that a Fubini type theorem for Haar meager sets fails.

Section~5 contains decomposition results for some special classes of Polish groups, in particular all Banach spaces or the group $\zet^\omega$: Such groups can be decomposed into a Haar meager set and a Haar null set. It is well known (and easy to see) that Euclidean spaces can be decomposed into a meager set and a set of measure zero; our results are the `non-locally compact versions' of this statement. 

Finally, in Section~6 we provide other examples showing the differences between the studied $\sigma$-ideals, this time with the help of the group of permutations of integers.



\section{Haar meager sets in non-abelian groups}

We are going to prove that the notion of a Haar meager set behaves well even in the nonabelian setting, namely that the collection of such sets forms a $\sigma$-ideal contained in the $\sigma$-ideal of meager sets. The proofs are of a similar nature as those given in the abelian case in \cite{Darji}. Since there does not necessarily exist a complete left invariant metric on a Polish group (e.g. $S_\infty$, see \cite[p.~8]{BeckerKechris}) we need to be more careful. In the proof of Theorem~\ref{sigmaideal} we are following the ideas from \cite{CohenKallman} where the authors corrected the wrong proof from \cite{Mycielski} that Haar null sets in a Polish group form a $\sigma$-ideal. 

\begin{definition}
\label{nonabelian definition}
Let $G$ be a Polish group. A set $A\subseteq G$ is said to be \emph{Haar meager} if there are a Borel set $B\subseteq G$ such that $A\subseteq B$, a compact metric space $K$, and a continuous mapping $f\colon K\to G$ such that $f^{-1}(gBh)$ is meager in $K$ for every $g,h\in G$.

We call $f$ a witness function for $A$
and $B$ is called a (Borel) hull of $A$.
We denote the collection of all Haar meager sets by $\mathcal{HM}$.
\end{definition}

\begin{lemma}\label{okolicko}
Let $G$ be a Polish group with a compatible metric $\rho$ and $L$ be a compact subset of $G$. Then for every $\varepsilon>0$ there is a neighborhood $U$ of $1$ such that $\rho(x\cdot u, x)<\varepsilon$ for every $x\in L$ and $u\in U$.
\end{lemma}

\begin{proof}
	By the continuity of the function $(x,u)\mapsto\rho(x\cdot u,x)$, for every $x\in L$ there are neighborhoods $V_x$ of $x$ and $U_x$ of $1$ such that the image of $V_x\times U_x$ is a subset of $[0,\varepsilon)$.
	Let $F\subseteq L$ be a finite set such that $L\subseteq\bigcup_{x\in F}V_x$.
	We define $U=\bigcap_{x\in F}U_x$.
	It is easy to check that this choice of $U$ works.	
\end{proof}

\begin{theorem}\label{sigmaideal}
Haar meager sets in a Polish group form a $\sigma$-ideal.
\end{theorem}

\begin{proof}
Let $G$ be a Polish group. Let us fix a compatible complete metric $\rho$. 
Clearly a subset of a Haar meager set is Haar meager. Thus it is enough to prove that the system of Haar meager sets is closed under countable unions. Suppose that $A_n$, $n\in\omega$, are Haar meager sets in $G$. For $n\in\omega$, there exists a Borel set $B_n\supseteq A_n$ and a continuous mapping $f_n\colon K_n\to G$ defined on a compact space $K_n$ such that $f_n^{-1}(gB_nh)$ is meager in $K_n$ for every $g,h\in G$. Without loss of generality we may suppose that $1\in f_n(K_n)$ for every $n\in\omega$.

Let $L_n=f_1(K_1)\cdot\ldots\cdot f_n(K_n)$. The set $L_n$ is compact in $G$ because it is a continuous image of $K_1\times\ldots\times K_n$ under the mapping $f_1\cdot \ldots \cdot f_n$.
For any fixed $n\in\omega$ we get by Lemma~\ref{okolicko} that there is a neighborhood $U_n$ of $1$ such that $\rho(x\cdot u, x)<2^{-n}$ for every $x\in L_n$ and $u\in U_n$.
Let $K_n'$ be the closure of $f_n^{-1}(U_n)$.

\begin{claim}
	For every $n\in\omega$, the mapping $f_n\rest_{K_n'}$ witnesses that $B_n$ is Haar meager.
\end{claim}

\begin{proof}
	Let us fix $n\in\omega$ and $g,h\in G$. The set $f_n^{-1}(U_n)$ is open in $K_n$ and $f_n^{-1}(gB_nh)$ is meager in $K_n$, and so $f_n^{-1}(gB_nh)\cap f_n^{-1}(U_n)$ is meager in $f_n^{-1}(U_n)$.
	Since each open subset of $G$ is comeager in its closure, the set $f_n^{-1}(gB_nh)\cap K_n'$ is meager in $K_n'$.
\end{proof}

Define $K=\prod_{n\in\omega} K_n'$ and $\varphi_n\colon K\to G$ by
\begin{equation*}
\varphi_n(x)=f_1(x_1)\cdot f_2(x_2)\cdot\ldots\cdot f_n(x_n).
\end{equation*}
All the mappings $\varphi_n$ are clearly continuous and $K$ is compact.
Moreover, by the choice of $U_n$ we obtain $\rho(\varphi_n(x), \varphi_{n+1}(x))\leq 2^{-n}$
and hence the sequence of mappings $\varphi_n$ converges uniformly.
Let $f=\lim \varphi_n$; then $f\colon K\to G$ is continuous, being the uniform limit of continuous mappings.

We claim that $f$ witnesses that $A=\bigcup_{n\in\omega} A_n$ is Haar meager. To that end, note that $B=\bigcup_{n\in\omega} B_n$ is Borel and it contains $A$, and take arbitrary $g,h\in G$ and $i\in\omega$; we aim to show that $f^{-1}(gB_ih)$ is meager. Fix $x_j\in K_j'$ for every $j\neq i$. Now,
\begin{IEEEeqnarray*}{rCl}
\IEEEeqnarraymulticol{3}{l}{\{x_i\in K_i'\colon f(x_1, \dots, x_i, \dots)\in gB_ih\}} \\\quad
& = & \{x_i\in K_i'\colon f_1(x_1)\cdot\ldots\cdot f_i(x_i)\cdot\ldots \in gB_ih\} \\\quad
& = & f_i^{-1}((f_1(x_1)\cdot\ldots\cdot f_{i-1}(x_{i-1}))^{-1}gB_ih(f_{i+1}(x_{i+1})\cdot\dots)^{-1})\cap K_i'
\end{IEEEeqnarray*}
is meager in $K_i'$ because it is the inverse image under $f_i\rest_{K_i'}$ of a translate of the set $B_i$. Hence, by the Kuratowski-Ulam theorem (see e.g. \cite[Theorem 8.41]{Kechris}) applied in the product space $( \prod_{j\neq i} K'_j) \times K'_i$, the Borel set $f^{-1}(gB_ih)$ is meager. Thus $f^{-1}(gBh)$ is meager for every $g, h\in G$, which implies that $A$ is Haar meager.
\end{proof}

In the proof of the next theorem we follow \cite{Darji} where the abelian case was proved.

\begin{theorem}
Every Haar meager set in a Polish group is meager.
\end{theorem}

\begin{proof}
Let $A$ be a Haar meager subset of a Polish group $G$, and let us have a Borel set $B\supseteq A$ together with a continuous mapping $f\colon K\to G$ witnessing the Haar meagerness of $A$.
Define $S=\{(g,x)\in G\times K\colon f(x)\in gB\}$; then $S$ is Borel.
For  any fixed $g\in G$ the set $\{x\in K\colon f(x)\in gB\}$ is meager in $K$ and it is a section of $S$.
Hence, by the Kuratowski-Ulam theorem, $S$ is meager in $G\times K$.
Using the Kuratowski-Ulam theorem again we conclude that there are comeager many $x\in K$ for which $\{g\in G\colon f(x)\in gB\}$ is meager. Since $K$ is compact, there is at least one such $x$. The last set is equal to $f(x)B^{-1}$, and can be mapped onto $B$ by an automorphism of the group $G$ and hence $B$ is meager as well. Thus also $A$ is meager.
\end{proof}

The following question has a positive answer in the abelian case (\cite[Corollary~1]{Jablonska}).

\begin{question}
Are compact sets in a non-locally compact Polish group Haar meager?
What about the special case $G=S_\infty$?
\end{question}


\section{Some notes on the definition of Haar meager sets}

It was shown in \cite{Darji} that in any abelian Polish group, the collection $\mathcal{HM}$ forms a $\sigma$-ideal (see \cite[Theorem 2.9]{Darji}) which is contained in the $\sigma$-ideal of all meager sets (see \cite[Theorem 2.2]{Darji}) and these two $\sigma$-ideals coincide if and only if $G$ is locally compact (see \cite[Theorem 2.4, Corollary 2.5 and Example 2.6]{Darji}).
Therefore we are mostly interested in the case where the group $G$ is not locally compact. 

It was explained in \cite[Example 2.3]{Darji} that some sort of definability condition (the Borelness of the set $B$ in the definition) is necessary to ensure that every Haar meager set is meager. However, it may be appealing to replace Borelness by the Baire property which seems to be better related to the Baire category. So let us denote by $\mathcal{\widetilde{HM}}$ the class of all sets satisfying Definition \ref{nonabelian definition} where we only require $B$ to have the Baire property (instead of being Borel). Even then, all relevant results from \cite{Darji} still remain true for the class $\mathcal{\widetilde{HM}}$ (with the same proofs). In particular, $\mathcal{\widetilde{HM}}$ is a $\sigma$-ideal which is contained in the $\sigma$-ideal of all meager sets, and these two $\sigma$-ideals coincide if and only if $G$ is locally compact. Nevertheless the following example shows that $\mathcal{\widetilde{HM}}$ is consistently a bigger $\sigma$-ideal than $\mathcal{HM}$.

\begin{example}
Assume the Continuum Hypothesis. Let $G$ be any abelian Polish group which is not locally compact. Then there is $A\subseteq G$ such that $A\in\mathcal{\widetilde{HM}}\setminus\mathcal{HM}$.
\end{example}

\begin{proof}
Let $K$ be an arbitrary perfect compact subset of $G$ and let $\{K_{\alpha}\colon\alpha<\omega_1\}$ be an ordering of all translates $gK$, $g\in G$, of the set $K$. 
Let $\{B_{\alpha}\colon\alpha<\omega_1\}$ be an ordering of all Borel subsets of $G$ which are in $\mathcal{HM}$.
By \cite[Corollary 1]{Jablonska}, each compact subset of $G$ is in $\mathcal{HM}$, and so we have
$B_{\alpha}\cup\bigcup\{K_{\beta}\colon \beta<\alpha\}\in\mathcal{HM}$ for every $\alpha<\omega_1$.
By \cite[Example 2.6]{Darji}, there is a closed meager subset $F$ of $G$ such that $F\notin\mathcal{HM}$.
Then for each $\alpha<\omega_1$, the set 
$F\setminus(B_{\alpha}\cup \bigcup\{K_{\beta}\colon \beta<\alpha\})$ 
is not in $\mathcal{HM}$, in particular it is nonempty.
So for every $\alpha<\omega_1$, we can pick
$p_{\alpha}\in F\setminus(B_{\alpha}\cup \bigcup\{K_{\beta}\colon \beta<\alpha\})$.
Finally, we put
$A=\{p_{\alpha}\colon\alpha<\omega_1\}$.

The set $A$ is meager since $A\subseteq F$, in particular $A$ has the Baire property. Moreover, $A$ intersects each translate of $K$ in at most countably many points, in particular in a meager subset of the translate. It easily follows that the identity mapping $id\colon K\rightarrow G$ witnesses that $A\in\mathcal{\widetilde{HM}}$. On the other hand, we have $A\notin\mathcal{HM}$ since $A$ is not contained in any $B_{\alpha}$, $\alpha<\omega_1$.
\end{proof}

In spite of the previous example, it is possible to replace Borelness by analyticity in Definition \ref{nonabelian definition}. This follows from the following proposition which is a straightforward modification of the proof of an analogous proposition from \cite{Solecki} (which deals with Haar null sets instead of Haar meager sets).

\begin{proposition}\label{P:HMforAnalytic}
Let $G$ be a Polish group and $A$ be an analytic subset of $G$. Let $K$ be a compact metric space and $f\colon K\rightarrow G$ be a continuous mapping such that $f^{-1}(gAh)$ is meager in $K$ for every $g,h\in G$. Then $A$ is Haar meager.
\end{proposition}

\begin{proof}
During this proof, whenever $Z \subseteq X\times Y$ (for any sets $X,Y$) and $x \in X$, we denote $Z_x=\{y\in Y\colon (x,y)\in Z\}$. Let
\begin{equation*}
\Phi = \{ X \subseteq G \colon X\in\Sigma_1^1 \text{ and } f^{-1}(gXh) \text{ is meager in } K \text{ for every } g,h\in G \}.
\end{equation*}
We verify that the family $\Phi$ is $\Pi_1^1$ on $\Sigma_1^1$, i.e. that for any Polish space $Y$ and any $\Sigma_1^1$ set $P\subseteq Y\times G$, the set $\{y \in Y\colon P_y\in\Phi\}$ is $\Pi_1^1$. To that end, let $Y$ be a Polish space and $P\subseteq Y\times G$ be $\Sigma_1^1$. Define $\tilde P\subseteq G\times G\times Y\times K$ by
\begin{equation*}
(g,h,y,x)\in\tilde P \Leftrightarrow x \in f^{-1}(gP_yh) \left(\Leftrightarrow (g^{-1}f(x)h^{-1},y) \in P\right).
\end{equation*}
Then $\tilde P$ is obviously $\Sigma_1^1$, and so by Novikov's theorem (see e.g. \cite[Theorem 29.22]{Kechris}), the set $\{(g,h,y)\in G\times G\times Y\colon \tilde P_{(g,h,y)} \text{ is meager in } K\}$ is $\Pi_1^1$. Therefore the set
\begin{equation*}
\{y\in Y\colon \tilde P_{(g,h,y)} \text{ is meager in } K \text{ for every } g,h\in G\}=\{y\in Y\colon P_y\in\Phi\}
\end{equation*}
is also $\Pi_1^1$.

Since $A\in\Phi$, it follows by (the dual form of) the First Reflection Theorem (see e.g. \cite[Theorem 35.10 and the remarks following it]{Kechris}) that there exists a Borel set $B\in\Phi$ such that $A\subseteq B$.
\end{proof}

We do not know whether the Borel hull from Definition \ref{nonabelian definition} of Haar meager sets can be replaced by an $F_{\sigma}$ hull (even in the abelian setting).

\begin{question}
	Is every Haar meager set contained in an $F_\sigma$ Haar meager set?
\end{question}

In \cite{ElekesVidnyanszky} it is shown that in every non-locally compact abelian Polish group, there is a Borel Haar null set which does not have a $G_{\delta}$ hull. This is a negative answer to what looks like an analogy of the previous question in case of Haar null sets.

The next natural question posed in \cite{Darji} asks whether it is true that for every Haar meager set $A$ in an abelian Polish group, there is a witness function $f$ for $A$ which is the identity function defined on some compact subset of the corresponding group.

\begin{question}[{\cite[Problem 2]{Darji}}]
Let $G$ be an abelian Polish group and $A\subseteq G$ be a Haar meager set. Is there a compact set $L\subseteq G$ such that $gA\cap L$ is meager in $L$ for every $g\in G$?
\end{question}

Although we do not know the answer, we do know that it is not possible to pick an arbitrary witness function $f\colon K\rightarrow G$ (where $K$ is a compact metric space) for $A$ and simply put $L=f(K)$. This is shown by the following example.

\begin{example}
There exist a $G_{\delta}$ Haar meager set $A\subseteq\er$, a compact metric space $K$ and a witness function $f\colon K\rightarrow\er$ for $A$ such that $A\cap f(K)$ is comeager in $f(K)$.
\end{example}

\begin{proof}
By \cite{Keleti}, there is a compact set $M \subseteq \er$ with Hausdorff dimension $1$ such that for every $z \in \er$, the intersection $(z+M) \cap M$ contains at most one point. Let $L$ be a perfect compact subset of $M$. Then for every $z \in \er$ the intersection $(z+L) \cap L$ also contains at most one point.
Let $\{U_i \colon i \in \omega\}$ be a basis of nonempty open subsets of $L$. For every $i \in \omega$, we fix a perfect set $C_i \subseteq U_i$ which is nowhere dense in $U_i$. We define
\begin{equation*}
K = \bigcup\limits_{\substack{n\in\omega\\ n\ge 1}}\left\{ \left(x,\tfrac{1}{n}\right) \colon x \in \bigcup_{i=0}^n C_i \right\} \cup \left\{ (x,0) \colon x \in L \right\} \subseteq L \times [0,1],
\end{equation*}
so that $K$ is a compact metric space. We also define $f \colon K \rightarrow \er$ to be the projection to the first coordinate, i.e. $f(x,t) = x$ for every $(x,t) \in K$.
Finally, we define
\begin{equation*}
A = L \setminus \bigcup_{i \in \en} C_i \subseteq \er,
\end{equation*}
which is obviously a $G_{\delta}$ subset of $\er$.

Let us check that $f$ witnesses that $A$ is Haar meager. Take an arbitrary $z \in \er$. If $z=0$ then
\begin{equation*}
f^{-1}(z+A) = f^{-1}(A) = \left\{ (x,0) \in K \colon x \in L \setminus \bigcup_{i \in \omega} C_i \right\} \subseteq L \times \{0\}.
\end{equation*}
And if $z \neq 0$ then (using the fact that $A\subseteq L=f(K)$) we have
\begin{equation*}
f^{-1}(z+A) = f^{-1}\left((z+A) \cap L\right) \subseteq f^{-1}\left((z+L) \cap L\right),
\end{equation*}
which is a preimage of either singleton or empty set. In both cases, we easily conclude that $f^{-1}(z+A)$ is nowhere dense in $K$.

On the other hand, $L \setminus A$ is the union of nowhere dense subsets $C_i$, $i\in\omega$, of $L$, and so $A$ is a comeager subset of $L = f(K)$.
\end{proof}

The remainder of this section is concerned with the notion of naively Haar meager sets:
\begin{definition}
Let $G$ be a Polish group.
A set $X\subseteq G$ is called \emph{naively Haar meager} if there is a continuous map $f$ of a metrizable compact space $K$ into $G$ such that $f^{-1}(gXh)$ is meager in $K$ for every $g, h\in G$.

The set $K$ is called a witness compact for $X$.
\end{definition}

The following Lemma is a topological analogue of the measure theoretic one \cite[Lemma~3.2]{ElekesVidnyanszky2015}.

\begin{lemma}\label{translation}
Let $G$ be an abelian Polish group.
Let $P\subseteq G$ be a perfect set and $C\subseteq G$ a comeager subset.
Then there exists $g\in G$ such that $|C\cap gP|=\mathfrak c$.
\end{lemma}

\begin{proof} 
Let $B$ be a Borel comeager subset of $C$. We will prove that $|B\cap gP|=\mathfrak c$ for some $g\in G$.
Suppose for a contradiction that there is no such $g\in G$. 
This means that $|B\cap gP|<\mathfrak c$ for every $g\in G$. Since Borel sets are either of size continuum or at most countable and the set $B\cap gP$ is Borel it follows that it is at most countable.
Hence also $gB\cap P$ is at most countable for every $g\in G$. Thus $gB\cap P$ is meager in $P$.
It follows that $B$ is Haar meager and thus by \cite{Darji} $B$ is also meager.
This is a contradiction, since $B$ was supposed to be comeager.
\end{proof}

The following is shown in \cite[Lemma~3.4]{ElekesVidnyanszky2015} to be an easy consequence of \cite{FarahSolecki}.

\begin{lemma}\label{subgroup}
Every uncountable Polish group has an uncountable Borel subgroup of uncountable index.
\end{lemma}

The following can be found in \cite[p.~5]{BeckerKechris}.

\begin{lemma}\label{perfect} 
Let $G$ be an abelian Polish group and let $H$ be a meager subgroup of $G$. Then there is a perfect partial transversal of $G/H$ (i.e. there is a perfect subset of $G$ which intersects each coset of $H$ in at most one point).
\end{lemma}

The next Theorem answers a question of Darji (\cite[Problem~1]{Darji}). We follow some ideas of \cite{ElekesVidnyanszky2015}.

\begin{theorem}
Let $G$ be an uncountable abelian Polish group.
Then there is a set $X\subseteq G$ which is naively Haar meager but not Haar meager.
\end{theorem}

\begin{proof}
Let $H$ be an uncountable Borel subgroup of $G$ with an uncountable index in $G$ (such a subgroup exists by Lemma~\ref{subgroup}). Since $H$ has the Baire property and is of an uncoutable index it is easy to see that it is meager. Thus by Lemma~\ref{perfect} there is a perfect partial transversal $P$ of $G/H$.
Since $H$ is uncountable and Borel there is a perfect set $K\subseteq H$ which will be used as a witness compact.
Note that if $T$ is an arbitrary partial transversal of $G/H$ then $T$ is naively Haar meager.

Let $\{F_\alpha\colon \alpha<\mathfrak c\}$ be the collection of all meager sets in $G$ which are of type $F_\sigma$. We will construct by induction a set $X=\{x_\alpha\colon \alpha<\mathfrak c\}\subseteq G$ such that 
\[x_\alpha\notin F_\alpha\cup \bigcup_{\beta<\alpha}x_\beta H.\]

To verify that this is possible it is enough to show that for a fixed $\alpha<\mathfrak c$ the set $F_\alpha\cup \bigcup_{\beta<\alpha}x_\beta H$ is a proper subset of $G$. Suppose that it equals $G$. Then the set $C=\bigcup_{\beta<\alpha}x_\beta H$ is comeager. Moreover $|C\cap gP|\leq |\alpha|<\mathfrak c$ for every $g\in G$. This is a contradiction with Lemma~\ref{translation}.

Since $X$ is a partial transversal it follows that it is naively Haar meager.
Moreover $X$ is not meager because it is not contained in any $F_\alpha$ for $\alpha<\mathfrak c$.
Hence $X$ is not Haar meager.
\end{proof}

It can be easily shown that under the Continuum Hypothesis, for every uncountable group $G$ there is a naively Haar meager set $X\subseteq G\times G$ such that $(G\times G)\setminus X$ is naively Haar meager as well. In fact it is enough to consider $X\subseteq G\times G$ as a well-order of $G$ which is of type $\omega_1$. It follows that all vertical sections of $X$ are countable and similarly for the horizontal sections of $G\setminus X$.
It follows that under the Continuum Hypothesis naively Haar meager sets do not form an ideal in the groups of the form $G\times G$. 

The reasoning from the preceding part is known in the context of naively Haar null sets \cite{ElekesVidnyanszky2015}.
In the following Proposition we are, however, able to show that this does not hold only for the squares of groups. Note that the Proposition can be analogously stated and proved also for naively Haar null sets.

\begin{proposition}
Assume the Continuum Hypothesis. Let G be an uncountable abelian Polish group. Then there exists a naively Haar meager set $X\subseteq G$ such that $G\setminus X$ is naively Haar meager as well. 
In particular, naively Haar meager sets do not form an ideal in $G$.
\end{proposition}

\begin{proof}
By Lemma~\ref{subgroup} there exists an uncoutable Borel subgroup $H\subseteq G$ of an uncountable index.
Since $H$ has the Baire property and its index in $G$ is uncountable it follows that $H$ is meager.
Thus by Lemma~\ref{perfect} there is a perfect partial transversal of $G/H$. Let us index $\{g_\alpha\colon \alpha<\omega_1\}=G$ and let $\{H_\alpha\colon \alpha<\omega_1\}$ be all the cosets of $H$ in $G$. Simply let
\[X=\bigcup_{\alpha<\omega_1}\left(g_\alpha P\setminus\bigcup_{\beta<\alpha}H_\beta\right).\]
If we fix an arbitrary perfect set $K\subseteq H$ it is a witness compact for $X$ (as $gX\cap K$ is at most countable for every $g\in G$) and thus $X$ is naively Haar meager. On the other hand for an arbitrary $g\in G$ there is $\alpha<\omega_1$ such that $g=g_{\alpha}$ and then 
\[(G\setminus X)\cap g P=g_\alpha P\setminus X\subseteq g_\alpha P\setminus \left(g_\alpha P\setminus\bigcup_{\beta<\alpha} H_\beta\right)= g_\alpha P\cap \bigcup_{\beta< \alpha} H_\beta.\]
Since the last set is countable it follows that any translation of $G\setminus X$ intersects $P$ in a countable and thus meager set. It follows that $G\setminus X$ is naively Haar meager as well.
\end{proof}



\section{Cartesian products and $\mathcal{HM}$}

The Cartesian products of (non-)Haar meager sets were already studied in \cite{Jablonska}. It was asked in \cite{Jablonska} whether the Cartesian product of two non-Haar meager Borel sets is non-Haar-meager. We affirmatively answer this question by proving Theorem \ref{T:Products} (note that the easier implication (ii)$\Rightarrow$(i) from Theorem \ref{T:Products} was already proved in \cite[Theorem 3]{Jablonska}). We remark that the question was stated for abelian groups only, but our theorem works also in the non-abelian setting.

\begin{theorem}
\label{T:Products}
Let $G_1,G_2$ be Polish groups, and let $A_1\subseteq G_1$ and $A_2\subseteq G_2$ be analytic sets. Then the following are equivalent:
\begin{enumerate}[(i)]
\item $A_1$ and $A_2$ are non-Haar meager in their respective groups;
\item $A_1\times A_2$ is non-Haar meager in $G_1\times G_2$.
\end{enumerate}
\end{theorem}

To prove Theorem \ref{T:Products}, we need a simple lemma:

\begin{lemma}\label{L:suitableshift}
Let $G$ be a Polish group and $A\subseteq G$ an analytic set which is non-Haar meager. let $K$ be a compact metric space and $f\colon K\to G$ be a continuous mapping. Then for any nonempty open set $U\subseteq K$ there exist $g, h\in G$ such that $f^{-1}(gAh)\cap U$ is non-meager.
\end{lemma}

\begin{proof}
Assume the contrary, i.e. there exist a compact metric space $K$, a continuous mapping $f\colon K\to G$, and an open set $U\subseteq K$ such that for each $g,h\in G$, $f^{-1}(gAh)\cap U$ is meager in $K$. It is then easy to see that  $(f\rest_{\overline{U}})^{-1}(gAh)$ is meager for any $g,h \in G$, and Proposition~\ref{P:HMforAnalytic} now gives that $A$ is Haar meager.
\end{proof}

\begin{proof}[Proof of Theorem \ref{T:Products}]
To prove implication (i)$\Rightarrow$(ii), let us observe that $A_1\times A_2$ is analytic, and so we can again use Proposition~\ref{P:HMforAnalytic} to free us of the obligation to consider Borel hulls. 

Take any $f\colon K\to G_1\times G_2$ continuous where $K$ is a compact metric space; we shall find points $(g_1,g_2), (h_1, h_2)\in G_1\times G_2$ such that $f^{-1}((g_1,g_2)(A_1\times A_2)(h_1,h_2))$ is non-meager in $K$, showing that $f$ does not witness Haar meagerness of $A_1\times A_2$.

Of course, there are continuous mappings $f_1\colon K\to G_1$ and $f_2\colon K\to G_2$ such that $f=(f_1,f_2)$. Since $A_1$ is not Haar meager in $G_1$, there exist $g_1, h_1\in G_1$ such that $f_1^{-1}(g_1A_1h_1)$ is non-meager in $K$. From the facts that $A_1$ is analytic and $f_1$ is continuous we obtain that $f_1^{-1}(g_1A_1h_1)$ is analytic; in particular it has the Baire property by the Lusin-Sierpi\'nski theorem (see e.g. \cite[Theorem 21.6]{Kechris}). Thus there is an open set $U\subseteq K$ such that $f^{-1}_1(g_1A_1h_1)\cap U$ is comeager in $U$. We now apply Lemma~\ref{L:suitableshift} to $A_2$, $f_2$ and $U$, obtaining $g_2, h_2\in G_2$ such that $f^{-1}_2(g_2A_2h_2)$ is non-meager in $U$. We have:
\begin{IEEEeqnarray*}{rCl}
& & f^{-1}((g_1,g_2)(A_1\times A_2)(h_1,h_2)) \\
& = & f^{-1}((g_1A_1h_1)\times(g_2A_2h_2)) \\ 
& = & f^{-1}((g_1A_1h_1)\times G_2) \cap f^{-1}(G_1\times (g_2A_2h_2)) \\
& = & f^{-1}_1 (g_1A_1h_1) \cap f^{-1}_2(g_2A_2h_2).
\end{IEEEeqnarray*}
The last set is clearly non-meager in $U$, so it is non-meager in $K$. 

To prove the opposite implication, let us assume that $A_1$ is Haar meager, take a witness function $f\colon K\to G_1$ and fix any point $h\in G_2$. Then one can readily see that the mapping $\hat{f}\colon K\to G_1\times G_2$ defined by $a\mapsto (f(a), h)$ witnesses the Haar meagerness of $A_1\times A_2$.
\end{proof}

Although Haar meagerness behaves nicely with respect to Cartesian product, a Fubini type theorem cannot hold as shown by the next example (which is known to work also for the $\sigma$-ideal of Haar null sets). 

The example was pointed out to one of the authors by M.~Elekes for which he has our thanks. The example seems to be rather standard, but we failed to find it in the literature, and so we state it for the sake of completeness. Note that the example works also for Haar null.

\begin{example}
There exists a non-Haar meager set $C$ in $\zet^{\omega}\times\zet^{\omega}$ such that in one direction all its sections are Haar meager, and in the other direction there are non-Haar meager many sections which are non-Haar meager.
\end{example}

\begin{proof}
Let us go straight to the definition of the desired set $C$:
\begin{equation*}
C = \{ (s,t)\in\zet^{\omega}\times\zet^{\omega} \colon t_n\le s_n\le 0 \text{ for every }  n\in\omega \}.
\end{equation*}
Now, take any $t\in\zet^{\omega}$; we want to observe that the corresponding horizontal section of $C$ is Haar meager. We have
\begin{equation*}
C\cap (\zet^{\omega}\times \{ t \}) = \{s\in\zet^{\omega}\colon t_n\le s_n\le 0 \text{ for every }  n\in\omega\}\times \{t\}.
\end{equation*}
But this set is easily seen to be compact (it is even empty if there is $n\in\omega$ such that $t_n>0$), and therefore it is also Haar meager (see \cite[Corollary 1]{Jablonska}).

To see our claim about the sections in the other direction, observe that the set of all (coordinatewise) non-positive $s\in\zet^{\omega}$ contains a translate of each compact set in $\zet^{\omega}$.
Therefore this set is not Haar meager and it suffices to show that for any (coordinatewise) non-positive $s\in\zet^{\omega}$, the corresponding section of $C$ is non-Haar meager.
We have
\begin{equation*}
C\cap(\{s\}\times\zet^{\omega})=\{s\}\times\{t\in\zet^{\omega}\colon t_n\le s_n \text{ for every }  n\in\omega\},
\end{equation*}
so the corresponding section of $C$, being a translate of the set of all (coordinatewise) non-positive $t\in\zet^{\omega}$, is not Haar meager.

Finally, let us prove that the set $C$ is not Haar meager by proving that $C$ contains a translation of each compact set.
For $(s,t)\in\zet^\omega\times \zet^\omega$, let us define the continuous projections
\begin{align*}
\pi^1_n(s,t) & = s_n, \text{ and} \\
\pi^2_n(s,t) & = t_n.
\end{align*}
Now, let $K\subseteq \zet^\omega\times \zet^\omega$ be compact; we claim that $(a,b)+K\subseteq C$ if we define for each $n\in\omega$, 
\begin{align*}
a_n = - \max \pi^1_n(K) \quad \text{and}\quad b_n = - \max \pi^2_n(K) - \diam \pi^1_n(K).
\end{align*}
Indeed, let $(\tilde{s},\tilde{t})\in K$ and set $s=\tilde{s}+a$ and $t=\tilde{t}+b$. We want to verify that $(s,t)\in C$. But from the definitions it is clear that for each $n\in\omega$ we have $s(n)\leq 0$ and
\begin{align*}
t(n) & =\tilde{t}(n)-\max\pi^2_n(K) - \diam \pi^1_n(K) \\ 
&\leq 0- (\max \pi^1_n(K) - \min \pi^1_n(K))\\
&\leq \tilde{s}(n)-\max \pi^1_n(K) \\
&= s(n).
\end{align*}
Therefore $(s,t)\in C$, and it follows that $(a,b)+K\subseteq C$.
\end{proof}

\section{Decompositions of groups into small sets}

By classical results, any uncountable locally compact abelian Polish group can be decomposed into two sets, one of them meager and the other Haar null (see e.g. \cite[Theorem 16.5]{Oxtoby}).
Recall that in the locally compact case, the family $\mathcal{HM}$ coincides with the $\sigma$-ideal of meager sets.
The following question was posed in \cite{Jablonska}.

\begin{question}[{\cite[Problem 4]{Jablonska}}]
Let $G$ be any abelian Polish group which is not locally compact. Can $G$ be decomposed into two disjoint sets $A$ and $B$ where $A$ is Haar meager in $G$ and $B$ is Haar null in $G$?
\end{question}

It was noted in \cite{Jablonska} that the answer is positive if $G$ is one of the sequence spaces $c_0$, $c$ or $l_p$, $p\geq 1$. Theorem \ref{generalization of sequence spaces} is a straightforward generalization of this observation. In Theorem \ref{clopen basis}, we show that such a decomposition is also possible in several other cases.

\begin{theorem}
	\label{generalization of sequence spaces}
	Let $G$ be a Polish group and $H\leq G$ be an uncountable closed subgroup of the center of $G$. Let $H$ be the union of a Haar null set in $H$ and a Haar meager set in $H$. Then $G$ is also the union of a Haar null set in $G$ and a Haar meager set in $G$.
	
	In particular, this holds for $G=\er^{\omega}$ or $G=X$ where $X$ is a Banach space.
\end{theorem}

\begin{proof}
	Note that without further reminders, we repeatedly use the fact that $H$ is a subset of the center of $G$.
	
There exist sets $A,B\subseteq H$ such that $A$ is Haar null in $H$, $B$ is Haar meager in $H$ and $A\cup B=H$. We can assume that $A$ and $B$ are Borel. Thus there exist a compact set $K$, a continuous function $f\colon K\to H$ and a Borel probability measure $\mu$ on $H$ such that for all $x\in H$ we have $\mu(xA)=0$ and $f^{-1}(xB)$ is meager in $K$. By \cite[Theorem 12.17]{Kechris} we can find a Borel set $M\subseteq G$ meeting every coset of $H$ in exactly one point. This means that $MH=G$ and that for $x,y\in M$, $x\neq y$ we have $xH\cap yH=\emptyset$. Thus the group operation is a continuous bijection of set $M\times H$ onto $G$. We set $\tilde{A}=MA$ and $\tilde{B}=MB$. By \cite[Theorem 15.1]{Kechris}, continuous injective images of Borel sets are Borel. Thus $\tilde{A}$ and $\tilde{B}$ are Borel. Clearly, $\tilde{A}\cup\tilde{B}=G$. Thus it remains to prove that $\tilde{A}$ is Haar null and $\tilde{B}$ is Haar meager. Let $y_1,y_2\in G$ be arbitrary. Then there exist $x_i\in H$ and $m_i\in M$ such that $y_i=m_ix_i$, $i=1,2$. Clearly, there exists exactly one $n\in M$ such that $y_1nHy_2\cap H\neq\emptyset$ and there exists $z\in H$ such that $nz=m_1^{-1}m_2^{-1}$. Thus $(y_1MAy_2)\cap H=(y_1nAy_2)\cap H=x_1x_2z^{-1}A$. So, 
\[\mu(y_1\tilde{A}y_2)=\mu((y_1MAy_2)\cap H)=\mu(x_1x_2z^{-1}A)=0.\] Similarly, 
\begin{eqnarray}\nonumber
	f^{-1}(y_1\tilde{B}y_2)=f^{-1}((y_1MBy_2)\cap H)=f^{-1}(x_1x_2z^{-1}B).
\end{eqnarray}
Thus $f^{-1}(y_1\tilde{B}y_2)$ is meager in $K$.
\end{proof}

For the rest of this section, we need to introduce some notation.

\begin{notation}
	\normalfont{
		We denote by $2^{<\omega}$ the set of all finite sequences of elements of $\{0,1\}$. For our purposes, it is convenient to exclude the empty sequence, so we suppose that each $s\in 2^{<\omega}$ has a strictly positive length, denoted by $l(s)$.
		For $s=(s_0,\ldots,s_{l(s)-1})\in 2^{<\omega}$ and $i\in\{0,1\}$, we write $s^{\wedge}i$ for the sequence $(s_0,\ldots,s_{l(s)-1},i)\in 2^{<\omega}$.
		For $\alpha=(\alpha_i)_{i=0}^{\infty}\in 2^{\omega}$ and $l\ge 1$, we denote by $\alpha\rest_l$ the restriction $(\alpha_0,\ldots,\alpha_{l-1})\in 2^{<\omega}$ of $\alpha$.
		For $s\in 2^{<\omega}$, we denote by $I_s$ the set $\{\alpha\in 2^{\omega}\colon \alpha\rest_{l(s)}=s\}\subseteq 2^{\omega}$.
		If $s\in 2^{<\omega}$ and $\alpha\in 2^{\omega}$ we write $s\subseteq\alpha$ if $s$ is an initial segment of $\alpha$.
		We use analogous notation for the sets $3^{<\omega}$ of all finite sequences of elements of $\{0,1,2\}$ and $\zet^{<\omega}$ of all finite sequences of elements of $\zet$.
		
		If $(G,+)$ is a group and $g\in G$ then by $0g$ we mean the identity element of $G$. If moreover $z',z''\in\zet$ are such that $z''<0<z'$ then by $z'g$ we mean 
		\begin{equation*}
		\underbrace{g+\ldots+g}_{z' \text{-times}}
		\end{equation*}
		and by $z''g$ we mean
		\begin{equation*}
			\underbrace{-g-\ldots-g}_{(-z'') \text{-times}}.
			\end{equation*}
		For $A\subseteq G$, we denote by $\left\langle A\right\rangle$ the smallest subgroup of $G$ containing every element of $A$.
		
		For $z,\tilde z\in\zet$, we write $\text{par}(z,\tilde z)$ as a shortcut for the statement `$z$ and $\tilde z$ have the same parity'.
	}
\end{notation}

Example \ref{priklad} is a particular case of Theorem \ref{clopen basis}. We still include it here as the main idea of the proof seems to be much more transparent in this simple setting.

\begin{example}\label{priklad}
	The group $\zet^{\omega}$ is the union of a Haar null set and a Haar meager set.
\end{example}

\begin{proof}
	Let $(s_n)_{n\in\omega}$ be an enumeration of the set $\zet^{<\omega}$. We define
	\begin{IEEEeqnarray*}{rCl}
		A^n_k & = & \{ x \in \zet^{\omega} \colon s_n \subseteq x \text{ and } x(i) \text{ is even}\\
		& & \text{for every }l(s_n) \leq i < l(s_n) + n + k \}, \ \ \ n,k\in\omega,\\
		A_k & = & \bigcup_{n \in \omega} A^n_k, \ \ \ k\in\omega,\\
		A & = & \bigcap_{k \in \omega} A_k.
	\end{IEEEeqnarray*}
	We will show that (a) $A$ is Haar null while (b) $\zet^{\omega} \setminus A$ is Haar meager.
	
	(a) Let $K =2^{\omega} \subseteq \zet^{\omega}$ and let $\mu$ be the natural `half-half' measure on $\zet^{\omega}$ supported by $K$. For an arbitrary $z \in \zet^{\omega}$ we want to show that $\mu(A + z) = 0$. To this end, it clearly suffices to show that for every $k \geq 1$ we have $\mu(A_k + z) \le \frac{1}{2^{k-1}}$. So fix $k \ge 1$. For every $n \in \omega$ and every $l(s_n) \leq i < l(s_n) + n + k$, the parity of $x(i)$ does not depend on the choice of $x \in A^n_k + z$. So we clearly have
	\begin{equation*}
		\mu(A^n_k + z) \leq \frac{1}{2^{n+k}},
	\end{equation*}
	and so
	\begin{equation*}
		\mu(A_k + z) \leq \sum\limits_{n \in \omega} \mu(A^n_k + z) \le \sum\limits_{n \in \omega} \frac{1}{2^{n+k}} = \frac{1}{2^{k-1}},
	\end{equation*}
	as we wanted.
	
	(b) Since Haar meager sets form a $\sigma$-ideal, it suffices to show that $\zet^{\omega}\setminus A_k$ is Haar meager for every $k\in\omega$. Let again $K = 2^{\omega} \subseteq \zet^{\omega}$ and fix $k\in\omega$ and $z \in \zet^{\omega}$ arbitrarily. We want to show that $\left( \left( \zet^{\omega} \setminus A_k \right) + z \right) \cap K$ is meager in $K$. The set $\left( \left( \zet^{\omega} \setminus A_k \right) + z \right) \cap K$ is a closed subset of $K$, and so it suffices to show that its complement $(A_k+z)\cap K$ is dense in $K$. So  let us fix $t \in 2^{<\omega}$; we want to find $x \in \left( A_k + z \right) \cap K$ such that $t \subseteq x$. Let $n\in\omega$ be such that $s_n = t - z\rest_{{l(t)}}$. Then we have
	\begin{multline*}
		A_k^n + z = \{ x\in\zet^{\omega} \colon t\subseteq x \text{ and par}(x(i),z(i)) \text{ for every }l(s_n) \leq i < l(s_n) + n + k \}.
	\end{multline*}
	Let us define $x \in \zet^{\omega}$ by
	\begin{equation*}
		x(i) =
		\begin{cases}
			t(i) & \text{if } i < l(t),\\
			0 & \text{if }l(s_n) \leq i < l(s_n) + n + k \text{ and } z(i) \text{ is even},\\
			1 & \text{if }l(s_n) \leq i < l(s_n) + n + k \text{ and } z(i) \text{ is odd},\\
			0 & \text{if } i \geq l(s_n) + n + k.
		\end{cases}
	\end{equation*}
	Then $t\subseteq x$ and $x \in \left( A_k^n + z \right) \cap K \subseteq \left( A_k + z \right) \cap K$, which concludes the proof.
\end{proof}

\begin{theorem}
	\label{clopen basis}
	Let $G$ be an abelian Polish group such that its identity element has a local basis consisting of open subgroups. Then $G$ is the union of a Haar null set and a Haar meager set.
\end{theorem}

\begin{proof}
	Let $\{G_n\colon n\in\omega\}$ be a local basis of the identity element of $G$ consisting of open subgroups such that $G=G_0\supsetneqq G_1\supsetneqq G_2\supsetneqq\ldots$.
	For each $n\in\omega$ we fix some $x_n\in G_n\setminus G_{n+1}$.
	By passing to a subsequence, we may assume that either
	\begin{equation}\label{binary}
	x_n+x_n\in G_{n+1}\text{ for every }n\in\omega
	\end{equation}
	or
	\begin{equation}\label{ternary}
	x_n+x_n\notin G_{n+1}\text{ for every }n\in\omega.
	\end{equation}
	Suppose first that (\ref{binary}) holds.
	For each $s=(s_0,\ldots,s_{l(s)-1})\in 2^{<\omega}$, we denote
	\begin{equation*}
		K_s=G_{l(s)}+\sum\limits_{i=0}^{l(s)-1}s_ix_i
	\end{equation*}
	and then we define
	\begin{equation*}
		K=\bigcap\limits_{l\ge 1}\;\bigcup\limits_{\substack{s\in 2^{<\omega}\\ l(s)=l}}\;K_s.
	\end{equation*}
	For each $s\in 2^{<\omega}$, the set $K_s$ is closed (this follows by the well-known fact that every open subgroup of a Polish group is closed). For every $s\in 2^{<\omega}$ and $i\in\{0,1\}$, we also have
	\begin{equation*}
		K_{s^{\wedge}i}=\left(G_{l(s)+1}+ix_{l(s)}\right)+\sum\limits_{i=0}^{l(s)-1}s_ix_i\subseteq G_{l(s)}+\sum\limits_{i=0}^{l(s)-1}s_ix_i=K_s.
	\end{equation*}
	Finally for any invariant compatible complete metric on $G$ (recall that such a metric exists since $G$ is an abelian Polish group) and for every $\alpha\in 2^{\omega}$, the diameter of $K_{\alpha\rest_l}$ tends to zero as $l\rightarrow\infty$.
	By these considerations it immediately follows that for each $\alpha\in 2^{\omega}$, the set $\bigcap_{l\ge 1}K_{\alpha\rest_l}$ contains precisely one element.
	So there is a unique mapping $\varphi\colon 2^{\omega}\rightarrow K$ such that $\varphi(\alpha)\in\bigcap_{l\ge 1}K_{\alpha\rest_l}$ for each $\alpha\in 2^{\omega}$. Then $\varphi$ is clearly a homeomorphism of the Cantor set $2^{\omega}$ onto $K$. In particular, $K$ is a compact set.
	
	For each $n\ge 1$, fix a selector $T_n$ of the subgroup $\left\langle G_n\cup\{x_{n-1}\}\right\rangle$ (i.e. fix $T_n\subseteq G$ such that it contains precisely one element from each coset of $\left\langle G_n\cup\{x_{n-1}\}\right\rangle$) and define
	\begin{equation*}
		S_n=G_n+T_n.
	\end{equation*}
	Note that $G$ is covered by (pairwise disjoint) cosets of $\left\langle G_n\cup\{x_{n-1}\}\right\rangle$. Note also that (1) easily implies that each coset of $\left\langle G_n\cup\{x_{n-1}\}\right\rangle$ is a disjoint union of precisely two cosets of $G_n$.
	So the cosets of $G_n$ are organized into pairs $\left(G_n+g,G_n+x_{n-1}+g\right)$, $g\in G$, and the union of the cosets from each such pair is a coset of $\left\langle G_n\cup\{x_{n-1}\}\right\rangle$. Obviously, precisely one coset of $G_n$ from each of the pairs intersects $T_n$. Since $S_n$ is the union of all cosets of $G_n$ intersected by $T_n$, it follows that $S_n$ contains one coset of $G_n$ from each of the pairs while being disjoint from the other. This is clearly true also for any set of the form $S_n+g$ where $g\in G$. 
	
	\begin{claim}\label{vzory}
		Let $g\in G$, $n\ge 1$, $s=(s_0,\ldots,s_{n-1})\in 2^n$ and $m\ge 1$. Then there is $t\in 2^{n+m}$ such that
		\begin{equation} \varphi^{-1}\left(K_s\cap\left(\bigcap_{i=n+1}^{n+m}S_i+g\right)\right)=I_t.
		\end{equation}
	\end{claim}
	\begin{proof}
		We proceed by induction on $m$.
		Suppose first that $m=1$. The set $S_{n+1}+g-\sum_{i=0}^{n-1}s_ix_i$ contains one of the sets $G_{n+1}$, $G_{n+1}+x_n$ and is disjoint from the other.
		In other words, the set $S_{n+1}+g$ contains one of the sets $K_{s^{\wedge}0}$, $K_{s^{\wedge}1}$ and is disjoint from the other.
		So $\varphi^{-1}(K_s\cap (S_{n+1}+g))$ equals either to $\varphi^{-1}(K_{s^{\wedge}0})=I_{s^{\wedge}0}$ or to $\varphi^{-1}(K_{s^{\wedge}1})=I_{s^{\wedge}1}$.
		
		Now suppose that $m>1$ and $\varphi^{-1}\left(K_s\cap\bigcap_{i=n+1}^{n+m-1}(S_i+g)\right)$ is of the form $I_t$ for some $t\in 2^{n+m-1}$.
		The set $S_{n+m}+g-\sum_{i=0}^{n+m-2}t_ix_i$ contains one of the sets $G_{n+m}$, $G_{n+m}+x_{n+m-1}$ and is disjoint from the other. In other words, the set $S_{n+m}+g$ contains one of the sets $K_{t^{\wedge}0}$, $K_{t^{\wedge}1}$ and is disjoint from the other.
		So the set
		\begin{equation*} \varphi^{-1}\left(K_s\cap\left(\bigcap_{i=n+1}^{n+m}S_i+g\right)\right)=I_t\cap\varphi^{-1}(S_{n+m}+g)
		\end{equation*}
		equals either to $\varphi^{-1}(K_{t^{\wedge}0})=I_{t^{\wedge}0}$ or to $\varphi^{-1}(K_{t^{\wedge}1})=I_{t^{\wedge}1}$.
	\end{proof}
	
	Let $\{B_n\colon n\in\omega\}$ be an enumeration of the open basis $\{G_m+g\colon m\ge 1, g\in G\}$ of $G$ (this basis is countable as it consists of cosets of countably many open subgroups).
	For each $n\in\omega$, let $\psi(n)\ge 1$ be such that $B_n$ is a coset of $G_{\psi(n)}$.
	Note that for each $n\in\omega$, the preimage $\varphi^{-1}(B_n)$ is either empty or of the form $I_s$ for some $s\in 2^{\psi(n)}$.
	We define
	\begin{IEEEeqnarray*}{rCl}
		A_k^n & = & B_n\cap\bigcap\limits_{i=\psi(n)+1}^{\psi(n)+n+k}S_i, \ \ \ n,k\in\omega,\\
		A_k & = & \bigcup\limits_{n\in\omega}A_k^n, \ \ \ k\in\omega, \\
		A & = & \bigcap\limits_{k\in\omega}A_k.
	\end{IEEEeqnarray*}
	We will show that (a) $A$ is Haar null while (b) $G \setminus A$ is Haar meager.
	
	(a) Let $\mu$ be the image of the product measure on $2^{\omega}$ under the homeomorphism $\varphi\colon 2^{\omega}\rightarrow K$. We consider $\mu$ as a measure on $G$ (supported by $K$). Let us fix $g \in G$ arbitrarily, we want to show that $\mu(A + g) = 0$. To this end, it clearly suffices to show that for every $k \geq 1$, we have $\mu(A_k + g) < \frac{1}{2^{k-1}}$. So let us fix $k \ge 1$. For each $n\in\omega$ let $k_n\in\omega$ be such that $B_n+g=B_{k_n}$.
	If $n\in\omega$ is such that $\varphi^{-1}\left(B_{k_n}\right)=\emptyset$, then also $\varphi^{-1}\left(A_k^n+g\right)\subseteq\varphi^{-1}(B_{k_n})=\emptyset$, and so $\mu(A^n_k+g)=0$. 
	Otherwise $\varphi^{-1}\left(B_{k_n}\right)$ is of the form $I_s$ for some $s\in 2^{\psi(k_n)}$, and so
	\begin{equation*} \varphi^{-1}\left(A_k^n+g\right)=I_s\cap\varphi^{-1}\left(\bigcap\limits_{i=\psi(k_n)+1}^{\psi(k_n)+n+k}S_i+
g\right)=\varphi^{-1}\left(K_s\cap\left(\bigcap\limits_{i=\psi(k_n)+1}^{\psi(k_n)+n+k}S_i+g\right)\right).
	\end{equation*}
	By Claim \ref{vzory}, this set is of the form $I_t$ for some $t\in 2^{\psi(k_n)+n+k}$. It immediately follows that 
\[\mu(A^n_k + g) \le \frac{1}{2^{\psi(k_n)+n+k}}<\frac{1}{2^{n+k}},\]
	and therefore
	\begin{equation*}
		\mu(A_k + g) \leq \sum\limits_{n \in \omega} \mu(A^n_k + g) < \sum\limits_{n \in \omega}\frac{1}{2^{n+k}} = \frac{1}{2^{k-1}}
	\end{equation*}
	as we wanted.
	
	(b) Since Haar meager sets form a $\sigma$-ideal, it suffices to show that $G\setminus A_k$ is Haar meager for every $k\in\omega$. Fix $k\in\omega$ and $g \in G$ arbitrarily. We want to show that $\left( \left( G \setminus A_k \right) + g \right) \cap K$ is meager in $K$. The set $\left( \left( G \setminus A_k \right) + g \right) \cap K$ is a closed subset of $K$, and so it suffices to show that its complement $(A_k+g)\cap K$ is dense in $K$. So  let us fix $s \in 2^{<\omega}$ arbitrarily, we want to show that $\left( A_k + g \right) \cap K \cap K_s\neq\emptyset$. Let $n\in\omega$ be such that $B_n=K_s-g$ (so that $\psi(n)=l(s)$); then we have
	\begin{equation*}
		A_k^n+g=K_s\cap\left(\bigcap\limits_{i=\psi(n)+1}^{\psi(n)+n+k}S_i+g\right).
	\end{equation*}
	By Claim \ref{vzory} it follows that there exists $t\in 2^{\psi(n)+n+k}$ such that
	\begin{equation*} \emptyset\neq\varphi(I_t)\subseteq(A^n_k+g)\cap K\cap K_s\subseteq(A_k+g)\cap K\cap K_s
	\end{equation*}
	as we wanted.
	
	Now suppose that (\ref{ternary}) holds.
	The proof is very similar in this case, we just have to use the Cantor set $3^{\omega}$ instead of $2^{\omega}$ (see Remark \ref{explanation} for an explanation).
	So for each $s=(s_0,\ldots,s_{l(s)-1})\in 3^{<\omega}$ we denote
	\begin{equation*}
		K_s=G_{l(s)}+\sum\limits_{i=0}^{l(s)-1}s_ix_i
	\end{equation*}
	and then we define
	\begin{equation*}
		K=\bigcap\limits_{l\ge 1}\;\bigcup\limits_{\substack{s\in 3^{<\omega}\\ l(s)=l}}\;K_s.
	\end{equation*}
	Similarly as above, there is a (unique) homeomorphism $\varphi\colon 3^{\omega}\rightarrow K$ of the Cantor set $3^{\omega}$ onto $K$ such that $\varphi(\alpha)\in\bigcap_{l\ge 1}K_{\alpha\rest_l}$ for each $\alpha\in 3^{\omega}$.
	
	For each $n\ge 1$ we do the following. Fix a selector $T_n$ of the subgroup $\left\langle G_n\cup\{x_{n-1}\}\right\rangle$.
	Let $z'_n$ be the smallest positive integer such that $G_n+z'_nx_{n-1}=G_n$, or let $z'_n=\infty$ if such a natural number does not exist. (Note that by (\ref{ternary}) we have $z'_n\ge 3$.)
	Next we define $z''_n=-1$ if $z'_n<\infty$ and $z''_n=-\infty$ if $z'_n=\infty$.
	It easily follows that each coset of $\left\langle G_n\cup\{x_{n-1}\}\right\rangle$ is a disjoint union of at least three (maybe infinitely many) pairwise disjoint cosets of $G_n$.
	If the considered coset of $\left\langle G_n\cup\{x_{n-1}\}\right\rangle$ is of the form $\left\langle G_n\cup\{x_{n-1}\}\right\rangle+g$ where $g\in T_n$, then the corresponding cosets of $G_n$ are $G_n+g+zx_{n-1}$, $z''_n<z<z'_n$.
	We define
	\begin{equation*}
		S_n=G_n+T_n+\bigcup\limits_{\substack{z''_n<z<z'_n\\ z\text{ is even}}}zx_{n-1}
	\end{equation*}
	(so that $S_n$ contains precisely those cosets of $G_n$ which are on the `even positions').
	
	\begin{claim}\label{ternary claim}
		Let $n\ge 1$ and $g\in G$.
		Then either one or two of the sets $G_n$, $G_n+x_{n-1}$ and $G_n+x_{n-1}+x_{n-1}$ are subsets of $S_n+g$, while the rest of these sets are disjoint from $S_n+g$.
	\end{claim}
	
	\begin{proof}
		It is clear that each of the sets $G_n$, $G_n+x_{n-1}$ and $G_n+x_{n-1}+x_{n-1}$ is either a subset of $S_n+g$ or is disjoint from $S_n+g$.
		So it suffices to show that at least one of them is a subset of $S_n+g$, and at least one of them is disjoint from $S_n+g$.
		Let $h$ be the unique element of $T_n\cap\left(\left\langle G_n\cup\{x_{n-1}\}\right\rangle-g\right)$.
		Then $G_n-g\subseteq\left\langle G_n\cup\{x_{n-1}\}\right\rangle+h$, and so
		there is $z''_n<z<z'_n$ such that $G_n-g=G_n+zx_{n-1}+h$. Then we have 
		\begin{IEEEeqnarray*}{lClll}
			G_n & = & G_n+h & +zx_{n-1} & +g,\\
			G_n+x_{n-1} & = & G_n+h & +(z+1)x_{n-1} & +g,\\
			G_n+x_{n-1}+x_{n-1} & = & G_n+h & +(z+2)x_{n-1} & +g.
		\end{IEEEeqnarray*}
		If $z+1<z'_n$ then one of the sets $G_n$ and $G_n+x_{n-1}$ is a subset of $S_n+g$ while the other is disjoint from $S_n+g$.
		Otherwise $z+1=z'_n$. Then $G_n+x_{n-1}=G_n+h+g$ and $G_n+x_{n-1}+x_{n-1}=G_n+h+x_{n-1}+g$, and so one of these two sets is a subset of $S_n+g$ while the other is disjoint from $S_n+g$.
	\end{proof}
	
	\begin{claim}\label{vzory2}
		Let $g\in G$, $n\ge 1$, $s=(s_0,\ldots,s_{n-1})\in 3^n$ and $m\ge 1$. Then
		\begin{equation*} \varphi^{-1}\left(K_s\cap\left(\bigcap_{i=n+1}^{n+m}S_i+g\right)\right)=\bigcup\limits_{t\in\mathcal T}I_t
		\end{equation*}
		where $\emptyset\neq\mathcal T\subseteq 2^{n+m}$ has at most $2^m$ elements.
	\end{claim}
	\begin{proof}
		We proceed by induction on $m$.
		Suppose first that $m=1$. By Claim \ref{ternary claim}, the set $S_{n+1}+g-\sum_{i=0}^{n-1}s_ix_i$ contains either one or two of the sets $G_{n+1}$, $G_{n+1}+x_n$ and $G_{n+1}+x_n+x_n$, and is disjoint from 
		the rest of these sets.
		In other words, the set $S_{n+1}+g$ contains either one or two of the sets $K_{s^{\wedge}0}$, $K_{s^{\wedge}1}$ and $K_{s^{\wedge}2}$, and is disjoint from the rest of these sets.
		So $\varphi^{-1}(K_s\cap (S_{n+1}+g))$ is equal either to one of the sets $I_{s^{\wedge}0}$, $I_{s^{\wedge}1}$ and $I_{s^{\wedge}2}$, or to the union of two of them.
		
		Now suppose that $m>1$ and
		\begin{equation*} \varphi^{-1}\left(K_s\cap\left(\bigcap_{i=n+1}^{n+m-1}S_i+g\right)\right)=\bigcup\limits_{t\in\mathcal T}I_t
		\end{equation*}
		where $\emptyset\neq\mathcal T\subseteq 2^{n+m-1}$ has at most $2^{m-1}$ elements.
		By Claim \ref{ternary claim}, we have for every $t\in\mathcal T$ that the set $S_{n+m}+g-\sum_{i=0}^{n+m-2}t_ix_i$ contains either one or two of the sets $G_{n+m}$, $G_{n+m}+x_{n+m-1}$ and $G_{n+m}+x_{n+m-1}+x_{n+m-1}$, and is disjoint from rest of these sets.
		In other words, the set $S_{n+m}+g$ contains either one or two of the sets $K_{t^{\wedge}0}$, $K_{t^{\wedge}1}$ and $K_{t^{\wedge}2}$, and is disjoint from the rest of these sets.
		So the set
		\begin{equation*} I_t\cap\varphi^{-1}\left(K_s\cap\left(\bigcap_{i=n+1}^{n+m}S_i+g\right)\right)=I_t\cap\varphi^{-1}(S_{n+m}+g)
		\end{equation*}
		is equal either to one of the sets $I_{t^{\wedge}0}$, $I_{t^{\wedge}1}$ and $I_{t^{\wedge}2}$, or to the union of two of them.
		Since this holds for every $t\in\mathcal T$, the conclusion easily follows.
	\end{proof}
	
	Let $\{B_n\colon n\in\omega\}$ be an enumeration of the open basis $\{G_m+g\colon m\ge 1, g\in G\}$ of $G$.
	For each $n\in\omega$, let $\psi(n)\ge 1$ be such that $B_n$ is a coset of $G_{\psi(n)}$. Then for each $n\in\omega$, the preimage $\varphi^{-1}(B_n)$ is either empty or of the form $I_s$ for some $s\in 3^{\psi(n)}$.
	As in the previous part of the proof, we again define
	\begin{IEEEeqnarray*}{rCl}
		A_k^n & = & B_n\cap\bigcap\limits_{i=\psi(n)+1}^{\psi(n)+n+k}S_i, \ \ \ n,k\in\omega,\\
		A_k & = & \bigcup\limits_{n\in\omega}A_k^n, \ \ \ k\in\omega, \\
		A & = & \bigcap\limits_{k\in\omega}A_k,
	\end{IEEEeqnarray*}
	and we will show that (a) $A$ is Haar null while (b) $G \setminus A$ is Haar meager.
	
	(a) Let $\mu$ be the image of the product measure on $3^{\omega}$ under the homeomorphism $\varphi\colon 3^{\omega}\rightarrow K$, considered as a measure on $G$ (supported by $K$). Let us fix $g \in G$ arbitrarily; we want to show that $\mu(A + g) = 0$. It suffices to show that for every $k \geq 1$ we have $\mu(A_k + g) < 3\cdot\left(\frac{2}{3}\right)^k$. So let us fix $k \ge 1$. For each $n\in\omega$ let $k_n\in\omega$ be such that $B_n+g=B_{k_n}$.
	If $n\in\omega$ is such that $\varphi^{-1}\left(B_{k_n}\right)=\emptyset$, then also $\varphi^{-1}\left(A_k^n+g\right)\subseteq\varphi^{-1}(B_{k_n})=\emptyset$, and so $\mu(A^n_k+g)=0$. 
	Otherwise $\varphi^{-1}\left(B_{k_n}\right)$ is of the form $I_s$ for some $s\in 3^{\psi(k_n)}$, and so
	\begin{equation*} \varphi^{-1}\left(A_k^n+g\right)=I_s\cap\varphi^{-1}\left(\bigcap\limits_{i=\psi(k_n)+1}^{\psi(k_n)+n+k}S_i
+g\right)=\varphi^{-1}\left(K_s\cap\left(\bigcap\limits_{i=\psi(k_n)+1}^{\psi(k_n)+n+k}S_i+g\right)\right).
	\end{equation*}
	By Claim \ref{vzory2}, this set is of the form $\bigcup\limits_{t\in\mathcal T}I_t$ where $\emptyset\neq\mathcal T\subseteq 2^{\psi(k_n)+n+k}$ has at most $2^{n+k}$ elements.
	It follows that 
\[\mu(A^n_k + g) \le 2^{n+k}\cdot\frac{1}{3^{\psi(k_n)+n+k}}<\left(\frac{2}{3}\right)^{n+k}\]
	and therefore
	\begin{equation*}
		\mu(A_k + g) \leq \sum\limits_{n \in \omega} \mu(A^n_k + g) < \sum\limits_{n \in \omega}\left(\frac{2}{3}\right)^{n+k} = 3\cdot\left(\frac{2}{3}\right)^k,
	\end{equation*}
	which is the desired conclusion.
	
	(b) Since Haar meager sets form a $\sigma$-ideal, it suffices to show that $G\setminus A_k$ is Haar meager for every $k\in\omega$. Fix any $k\in\omega$ and $g \in G$. We want to show that $\left( \left( G \setminus A_k \right) + g \right) \cap K$ is meager in $K$. The set $\left( \left( G \setminus A_k \right) + g \right) \cap K$ is a closed subset of $K$, and so it suffices to show that its complement $(A_k+g)\cap K$ is dense in $K$. So  let us fix an arbitrary $s \in 3^{<\omega}$; we want to show that
    $\left( A_k + g \right) \cap K \cap K_s\neq\emptyset$. Let $n\in\omega$ be such that $B_n=K_s-g$ (so that $\psi(n)=l(s)$). Then we have
	\begin{equation*}
		A_k^n+g=K_s\cap\left(\bigcap\limits_{i=\psi(n)+1}^{\psi(n)+n+k}S_i+g\right),
	\end{equation*}
	and by Claim \ref{vzory2}, it follows that there is $t\in 3^{l(s)+n+k}$ such that
	\begin{equation*} \emptyset\neq\varphi(I_t)\subseteq(A^n_k+g)\cap K\cap K_s\subseteq(A_k+g)\cap K\cap K_s.
	\end{equation*}
	This completes the proof.
\end{proof}

\begin{remark}
	\label{explanation}
	\normalfont{
		The reason why we used the Cantor set $3^{\omega}$ instead of $2^{\omega}$ in the second part of the proof of Theorem \ref{clopen basis} is that the `binary analogy' of Claim \ref{ternary claim} does not hold (and thus Claim \ref{vzory2} is a little more complicated that Claim \ref{vzory}).
		Indeed, it could happen e.g. that $g_0+g_0+g_0\in G_1$ and $G_0=\left\langle G_1\cup\{x_0\}\right\rangle$, so that $G=G_0$ is the union of three (pairwise disjoint) cosets of $G_1$. (The obstacle is hidden in the degree three, and any finite odd degree would cause similar problems.)
		The cosets $G_1$ and $G_1+x_0+x_0$ intersect $S_1$ (while $G_1+x_0$ is disjoint from $S_1$).
		Then $S_1+x_0$ is disjoint neither from $G_1$ nor from $G_1+x_0$.
		Of course, this phenomenon could occur also at higher levels of the construction.
		
		Of course, a straightforward application of the proof of Theorem \ref{clopen basis} to the group $\zet^{\omega}$ would lead to the branch of the proof where \eqref{ternary} holds, and so we would have to use the Cantor set $3^\omega$. However, as it is shown in Example \ref{priklad}, in this particular case the `binary' Cantor set works as well, and seems to be more natural.
	}
\end{remark}


\section{Group of permutations}


In this last section, we provide a natural example of a set which is meager but not Haar meager, and of a set which is Haar null but not Haar meager. This is motivated by Problem 4 and Problem 5 from \cite{Darji}.

 By \cite{DoughertyMycielski}, a typical permutation with respect to the ideal of Haar null sets contains infinitely many infinite cycles and finitely many finite cycles. The authors moreover proved that their result is the best possible in a certain sense.


\begin{example}
\normalfont
Let $S_\infty$ denote the group of permutations of natural numbers equipped with the group operation given by $pq=p\circ q$ for $p,q\in S_\infty$.

It follows by \cite{DoughertyMycielski} that the set $F$ of all permutations with only finite cycles is Haar null. On the other hand the set $F$ is easily seen to be comeager (as also noted in \cite{DoughertyMycielski}) and hence it is not Haar meager.
Also the set $P=S_\infty\setminus F$ is meager but as we will prove it is not Haar meager. The reason is that it contains a translation of an arbitrary compact subset of $S_\infty$. 

Note that $P$ is the set of all permutations in $S_\infty$ with an infinite cycle.
We claim that $P$ contains a shift of any compact set in $S_\infty$.
Let $K$ be a nonempty compact set in $S_\infty$.
The $n$-th projection $\pi_n(p)=p(n)$ to the $n$-th coordinate is continuous and thus $\pi_n(K)$ is finite.
Also the mapping $\rho_n(p)=p^{-1}(n)$ is continuous and thus $\rho_n(K)$ is finite.
We denote $a_n=\min \pi_n(K)$ and $b_n=\max\pi_n(K)$.
It follows that $a_n$ as well as $b_n$ converge to infinity.
Let $A_n=\min\{a_n, a_{n+1},\dots\}$, $B_n=\max\{b_1,\dots, b_n\}$.
Again $A_n$ and $B_n$ converge to infinity.
(Roughly speaking the compact set $K$ consists of those permutations whose graphs lie between the nondecreasing sequences $A_n$ and $B_n$.)

We are going to construct a permutation $p\in S_\infty$ such that $pK\subseteq P$.
By induction we are going to define sequences of positive integers $x_n, x_n', y_n, y_n'$, such that 
\begin{align*}
x_n&\leq x_n'\\
y_n&\leq y_n'\\
x_n'-x_n&=y_n'-y_n\\
x_{n+1}&\geq x_n'+2\\
y_{n+1}&\geq y_n'+2 
\end{align*}
such that $A_{y_{n}}\geq x_n'+2$.
Let $x_1=1, x_1'=B_1$.
Let $y_1$ be big enough such that $A_{y_1}\geq x_1'+2$.
Let $y_1'=y_1-x_1+x_1'$.

Suppose that  $x_n, x_n', y_n, y_n'$ have been constructed. We define
$x_{n+1}=A_{y_n}$ and $x_{n+1}'=B_{y_n'}$.
We define $y_{n+1}$ such that $A_{y_{n+1}}\geq x_{n+1}'+2$ and $y_{n+1}\geq y_n'+2$. We put
$y_{n+1}'=y_{n+1}-x_{n+1}+x_{n+1}'$.

All the intervals $[x_n, x_n']$ are pairwise disjoint.
All the intervals $[y_n, y_n']$ are pairwise disjoint.
Let us denote $X_n=[x_{n}, x_{n}']$, $Y_n=[y_{n}, y_{n}']$
We define partially the permutation $p$ in such a way that it maps $X_n$ to $Y_n$ bijectively for every $n\ge 1$. By the properties of the sequences $x_n, x_n', y_n, y_n'$, we have infinitely many gaps in the domain (e.g. the points $x_n'+1$) and also in the range of $p$ (e.g. the points $y_n'+1$) so we can enlarge the partial bijection $p$ to a bijection of the whole set $\N$.

Observe that by the construction for every $q\in K$ we get that $q(k)\in X_{n+1}$ for $k\in Y_n$.
Hence we have
$q(1)\in X_1$, $pq(1)\in Y_1$, $qpq(1)\in X_2$, $pqpq(1)\in Y_2$, $\dots$.
Since the sets $Y_n$ are pairwise disjoint it follows that the permutation $pq$ has an infinite cycle.
Hence $pK\subseteq P$.

\end{example}

\bibliographystyle{plain}
\bibliography{References}

\end{document}